\documentclass[10pt]{article}
\usepackage{amssymb,amsmath,amsthm,amsfonts}
\usepackage{cite}
\usepackage[colorlinks=true,linkcolor=blue,urlcolor=blue,citecolor=blue]{hyperref}
\usepackage[margin=100pt]{geometry}

\usepackage[pagewise,mathlines]{lineno}
\usepackage{fancyhdr}
\tt
\theoremstyle{plain}
\newtheorem{theorem}{Theorem}
\newtheorem{lemma}{Lemma}
\newtheorem{proposition}{Proposition}
\newtheorem{corollary}{Corollary}
\theoremstyle{definition}
\newtheorem{definition}{Definition}
\theoremstyle{remark}
\newtheorem{remark}{Remark}

\begin{document}

\title{\bf\Large Optimal location of resources and Steiner symmetry\\ in a population
dynamics model in heterogeneous environments}

\author{Claudia Anedda\footnote{Department of Mathematics and Computer Science,
University of Cagliari, Via Ospedale 72, Cagliari, 09124,  Italy (\tt canedda@unica.it).}\; and\;
Fabrizio Cuccu\footnote{Department of Mathematics and Computer Science,
University of Cagliari, Via Ospedale 72, Cagliari, 09124,  Italy (\tt fcuccu@unica.it).}}

\maketitle
\begin{abstract}
The subject of this paper is inspired by \cite{CC} and \cite{CCP}. In  \cite{CC} the authors 
investigate the dynamics of  a population in a heterogeneous environment by means of
diffusive logistic equations. An important part of their study consists in finding sufficient
conditions which guarantee the survival of the species. Mathematically, this task leads to
the weighted eigenvalue problem $-\Delta u =\lambda m u $ in a bounded smooth domain
$\Omega\subset \mathbb{R}^N$, $N\geq 1$, under homogeneous Dirichlet boundary
conditions, where $\lambda \in \mathbb{R}$ and $m\in L^\infty(\Omega)$.
The domain $\Omega$ represents the environment  and $m(x)$, called the local growth
rate, says where the favourable
and unfavourable habitats are located. Then, the authors in \cite{CC} consider a class
of weights $m(x)$ corresponding to environments where the total sizes of favourable and
unfavourable habitats are fixed, but their spatial arrangement is allowed to change; they
determine the best choice among them for the population to
survive.\\
In our work we consider a sort of refinement of the result above.
We write the weight $m(x)$ as sum of two (or more) terms, i.e. $m(x)=
f_1(x)+f_2(x)$, where $f_1(x)$ and $f_2(x)$ represent the spatial densities of the two
resources which contribute to form the local growth rate $m(x)$. Then, we fix the 
total size of each resource allowing their spatial location to vary. As our first main result,
we show that there exists an optimal choice of $f_1(x)$ and $f_2(x)$.
Our proof relies on some results in  \cite{CCP} and on a new property (to our knowledge)
about the classes of rearrangements of functions.
Moreover, we show that if $\Omega$ is Steiner symmetric, then the best arrangement of the 
resources inherits the same kind of symmetry (actually, this is proved in the more 
general context of the classes of rearrangements of measurable functions).
\end{abstract}



\section{Introduction and main results}\label{intro}
\noindent In this paper we consider the weighted eigenvalue problem
\begin{linenomath}
\begin{equation}\label{A}
 \begin{cases}
 -\Delta u =\lambda m(x) u \quad &\text{in }\Omega,\\
 u=0 & \text{on } \partial\Omega,
 \end{cases}  
\end{equation}
\end{linenomath}
where $\Omega\subset\mathbb{R}^N$, $N\geq 1$,  is a bounded smooth domain with boundary $\partial\Omega$, the weight $m(x)$ belongs to $L^\infty(\Omega)$
and $\lambda\in\mathbb{R}$.
When the set $\{x\in\Omega:m(x)>0\}$ (respectively, $\{x\in\Omega:m(x)<0\}$) has
positive Lebesgue measure, problem \eqref{A} admits a increasing (decreasing)  sequence 
of 
positive (negative) eigenvalues (see \cite{dF}).
Here, under the first assumption above, we are interested in the  smallest positive eigenvalue $\lambda_1(m)$, which we will call the {\it principal positive eigenvalue} (see Section \ref{preliminaries}).    
More precisely, we study the minimization of $\lambda_1(m)$ when $m$ is chosen in an appropriate class of bounded measurable functions.

\noindent Problem \eqref{A} originates from the study of reaction-diffusion equations in mathematical ecology which date to the pioneering work \cite{S} of Skellam. Precisely, we deal with the following model examined by Cantrell and Cosner in \cite{CC} and \cite{CC91}
\begin{linenomath}
\begin{equation}\label{p1}
\begin{cases}
v_t=d\Delta v+[m(x)-cv]v \quad &\text{in }\Omega\times(0,\infty),\\
v(x,0)=v_0(x)\geq 0 & \text{for } x\in\overline{\Omega},\\
v(x,t)=0 & \text{on } \partial\Omega\times(0,\infty).
\end{cases}  
\end{equation}
\end{linenomath}
In \eqref{p1} $v(x,t)$ represents the population density of a species inhabiting the region $\Omega$ in position $x$ at time $t$ surrounded by the hostile region $\mathbb{R}^N\smallsetminus\Omega$, $v_0$ is the initial density and $c, d$ are positive constants describing the limiting effects of crowding and the diffusion rate of the population, respectively. The function $m(x)$ represents the local grow rate of
the population, it is positive on favourable habitats and negative on unfavourable ones. 
Moreover, the homogeneous
Dirichlet boundary conditions mean that the exterior of $\Omega$
is a deadly environment (any individual reaching the boundary dies).
The aim of the papers \cite{CC} and \cite{CC91} is to study how the spatial arrangements of favourable and unfavourable habitats in $\Omega$ affects the survival of the modelled population.
The authors show that \eqref{p1} predicts persistence for the population if $\lambda_1(m)<1/d$. As a consequence, determining the best spatial arrangement of favourable and
unfavourable habitats to survival, within a fixed class of environmental configurations, results in minimizing $\lambda_1(m)$ over the corresponding class of weights.
In \cite{CC} and \cite{CC91} different aspects of model \eqref{p1} has been extensively studied.
This kind of problems has been investigated by many authors,
mostly under Neumann boundary conditions (in this case the
boundary $\partial\Omega$ acts as a fence and any individual
reaching it returns to $\Omega$) or in the case of a periodic 
environment ($\Omega=\mathbb{R}^N$ and $m(x)$ periodic).
In particular, we mention Berestycki et al. (\hspace{-0.01cm}\cite{BHR}) and Lou and Yanagida (\hspace{-0.01cm}\cite{LY}), who investigated how the fragmentation of the environment affects the persistence of the population, Roques and Hamel (\hspace{-0.01cm}\cite{RH}), who studied the optimal arrangement of resources by using numerical computation, Jha and Porru (\hspace{-0.01cm}\cite{JP}) who, among the other things, exhibited an example of symmetry breaking of the optimal arrangement of
the local growth rate, Cosner et al. (\hspace{-0.01cm}\cite{CCP}),
who considered the minimization of $\lambda_1(m)$ in the framework of rearrangements of functions and Lamboley et al. (\hspace{-0.01cm}\cite{LLNP}), who investigated model \eqref{p1} with Robin boundary conditions.\\ 
Our present study has two main aims, which we describe in details in what follows.
 Throughout the paper we consider the Lebesgue measure on $\mathbb{R}^N$ and denote by $|E|$  the measure of a measurable set $E$. Moreover, $\chi_E$ represents the usual
characteristic or indicator function of the set $E$.\\
In \cite[Theorem 3.9]{CC} the authors consider the minimization of $\lambda_1(m)$ when $m$ is taken in the set
\footnote{We follow the notation of \cite[Theorem 3.9]{CC} except for the replacement of $m_0$
by $m_3$, $\bar{m}$ by $\check{m}$ and $\lambda_1^+$ by $\lambda_1$.}
\begin{linenomath}
\begin{equation*}
\begin{split}
 \mathcal{M}=\ &\Bigg\{m(x)\in L^\infty(\Omega):
 -m_2\leq m(x)\leq m_1\text{ a.e. in }\Omega, \int_\Omega m\,dx=m_3
 \\ & \text{ and } m(x)>0 \text{ on a set of positive measure}
\Bigg\},
\end{split} 
\end{equation*}
\end{linenomath}
where $m_1,m_2$ and $m_3$ are suitable fixed constants with $m_1$ and $m_2$
positive.
The choice of the class $\mathcal{M}$ has the following biological meaning: the local growth rate of the population in each point of $\Omega$ has a value between 
the minimum $-m_2$ and the maximum $m_1$;
moreover, the ``total growth rate'' is fixed and equal to $m_3$. This can be
obtained by introducing different amounts of  resources in 
each point of the environment $\Omega$; here, in order to simplify our exposition,
by the term ``resource'' we mean anything that affects the growth rate of the population both in positive (for example food) and in negative (for example predators) sense. How
should we arrange the favourable and unfavourable habitats in $\Omega$ in order to
maximize the chance of survival of the population? Theorem
3.9 in \cite{CC} establishes that an optimal choice consists in using only the maximum
$m_1$ and the minimum $-m_2$ values of the local growth rate.
In terms of the weight $m(x)$, this means that
an optimal configuration is a distribution of ``bang-bang''
type, i.e. is a function which takes only two values in $\Omega
$. \\
We note that in the previous approach the maximization of the chance of  
survival of the population is considered with  constraints only on the local growth rate
$m(x)$ and not on the resources available in the environment $\Omega$. Indeed, the same
local growth rate can be obtained with different blends and/or number of resources. 
For example, the same effect can be attained putting
some amount of food in $\Omega$ or putting more food but introducing also some
predators in the environment. In our work we decompose the local growth rate $m(x)$ into
a sum of two terms $f_1(x)$ and $f_2(x)$ which we interpret as the spatial densities of
different types of resources. Then, we fix the total amount of each resource and seek for
their best location in $\Omega$ to maximize the chance of the population to
survive. In mathematical terms, we consider the minimization of $\lambda_1(m)$ as 
the  weight $m(x)$ varies in the class
\begin{linenomath}
\begin{equation*}
\begin{split}
 \mathcal{M}=\,&\{m(x)\in L^\infty(\Omega): m(x)=f_1(x)
+f_2(x),\, f_1\in
\mathcal{F}_1,\,f_2\in\mathcal{F}_2,\\
& m(x)>0 \text{ on a set of positive measure}\},
\end{split} 
\end{equation*}
\end{linenomath}
where
\begin{linenomath}
$$\mathcal{F}_1=\left\{f(x)\in L^\infty(\Omega): -p_1\leq f(x)\leq 
q_1\text{ a.e. in }\Omega\text{ and }\int_\Omega f
\,dx=l_1\right\}$$
\end{linenomath}
and
\begin{linenomath}
$$\mathcal{F}_2=\left\{f(x)\in L^\infty(\Omega): -p_2\leq  f(x)\leq 
q_2\text{ a.e. in }\Omega\text{ and }\int_\Omega f\,dx=l_2\right\},$$
\end{linenomath}
with $p_1$, $q_1$, $l_1$ and $p_2$, $q_2$, $l_2$ suitable 
fixed constants.
As our first main result, we prove that a minimizer actually
exists.  As it can be expected, each component $f_1(x)$ and 
$f_2(x)$ of the optimal weight must be of ``bang-bang'' type.
Moreover, we find that the two ``optimal components''  cannot
be chosen independently. Indeed, they have a special feature: $f_2(x)$ has to be large
(respectively, small) where $f_1(x)$ is large (small) as much as
possible. To be precise, $f_1(x)$ and $f_2(x)$ must realize the equality sign in the Hardy-Littlewood
inequality \eqref{HLr}. 
We get a complete description of our minimizer in the following
theorem.
\begin{theorem}\label{cantrell2}
Let $\Omega\subset\mathbb{R}^N$ be a bounded smooth domain and $\mathcal{M}=\{m(x)\in L^\infty(\Omega): m(x)=f_1(x)
+f_2(x),\, f_1\in
\mathcal{F}_1,\,f_2\in\mathcal{F}_2,\, m(x)>0 \text{ on a set of positive measure}\}$, where
$\mathcal{F}_i=\{f(x)\in L^\infty(\Omega): -p_i \leq f(x)\leq 
q_i\text{ a.e. in }\Omega\text{ and }\int_\Omega f\,dx=l_i\}$, with
$p_i,q_i$ and $l_i$ constants such that
$-p_i|\Omega|< l_i<q_i|\Omega|$, $i=1,2$, and $q_1+q_2>0$. Moreover, let $
e_i=(p_i|\Omega|+l_i)/ (p_i+q_i)$, $i=1,2$.
Then there exist two measurable sets $E,G\subseteq\Omega$, with $|E|=e_1,|G|=e_2$,
subject to the conditions
\begin{linenomath}
\begin{equation*}
\begin{cases}
E\supset G &\text{if }e_1>e_2,\\
E=G & \text{if } e_1=e_2,\\
E\subset G &\text{if }e_1<e_2,
\end{cases}  
\end{equation*}
\end{linenomath}
such that $\check f_1=q_1\chi_E-p_1\chi_{\Omega\smallsetminus E}\in
\mathcal{F}_1$, $\check f_2=q_2\chi_G-p_2\chi_{\Omega\smallsetminus G}\in
\mathcal{F}_2$ and $\check m=\check f_1+\check f_2\,(\in \mathcal{M})$
satisfies $\lambda_1(\check{m})=\inf\{\lambda_1(m):m\in\mathcal{M}\}$.
\end{theorem}
\noindent
The proof of Theorem \ref{cantrell2} relies on some results contained in \cite{CCP}
and on a new result (to our knowledge) about the classes of rearrangements of functions 
(see Section \ref{somme}). Incidentally, we note that, using the same theoretical
machinery, Theorem 3.9 in \cite{CC} can easily be proved.\\
Actually, the statement of Theorem \ref{cantrell2} can be
strengthen. Indeed, as a matter of fact, every minimizer of 
$\lambda_1(m)$  in the class $\mathcal{M}$ 
has the form described in Theorem \ref{cantrell2}. Unfortunately, this
does not follow from \cite{CCP} as it happens for Theorem \ref{cantrell2}; instead, it 
can be shown reasoning exactly as in the proof of Theorem
1.1 in \cite{ACF}. Due to the length of this argument, we prefer to
include this topic in a future paper.\\
We also note that Theorem \ref{cantrell2} can  easily be extended to
the general case
\begin{linenomath}
\begin{equation*}
\begin{split}
 \mathcal{M} =\, &\{m(x)\in L^\infty(\Omega): m(x)=\sum_{i=1}
 ^nf_i(x),\, f_i\in
\mathcal{F}_i\\
&\text{ and } m(x)>0 \text{ on a set of positive measure}\},
\end{split} 
\end{equation*}
\end{linenomath}
where
\begin{linenomath}
$$\mathcal{F}_i=\left\{f(x)\in L^\infty(\Omega): -p_i\leq f(x)\leq 
q_i\text{ a.e. in }\Omega\text{ and }\int_\Omega f\,dx=l_i\right\},$$
\end{linenomath}
with $i=1,\ldots,n$. In this case, there exists a minimizer $\check{m}$ which takes at least
two and at most $n+1$ 
values in $\Omega$.\\
The second main result of this paper concerns the Steiner 
symmetry of all the minimizers $\check m$ of $\lambda_1(m)$, when
$m$ varies in a class of rearrangements of a fixed bounded function.
Two measurable functions $f,g:\Omega \to \mathbb{R}$
are said \emph{equimeasurable} if the superlevel sets
$\{x\in \Omega: f(x)>t\}$ and $\{x\in \Omega: g(x)>t\}$ have the same  measure for all $t\in \mathbb{R}$. For a fixed $f\in L^\infty(\Omega)$ we call the set
$\mathcal{G}(f)=\{g:\Omega\to\mathbb{R}: g \text{ is measurable and $g$ and
$f$ are equimeasurable}\}$ the
\emph{class of rearrangements of $f$} (see Subsection
 \ref{rearra}).
Roughly speaking,
a set is {\it Steiner symmetric} if it is symmetric and convex relative to a hyperplane and a function is Steiner symmetric if any of its superlevel set
is Steiner symmetric (see Section \ref{secsteiner}).

\noindent In what follows we denote a point $x\in\mathbb{R}^N$ by $(x_1,x')$, 
where $x_1\in\mathbb{R}$ and $x'\in\mathbb{R}^{N-1}$.

\begin{theorem}\label{steiner}
Let $\Omega\subset\mathbb{R}^N$ be a bounded smooth domain and  assume it is Steiner symmetric with respect to the hyperplane
$T=\{x=(x_1,x')\in \mathbb{R}^N: x_1=0\}$. Let ${m}_0\in L^\infty(\Omega)$ such that $\{x\in \Omega: m_0(x)>0\}$ has
positive measure. Then every minimizer $\check m$ of the problem 
\begin{linenomath}
\begin{equation}\label{teo2}
\inf\{\lambda_1({m}):{m}
\in\mathcal{G}({m}_0)\}
\end{equation}
\end{linenomath}
is Steiner symmetric relative to $T$. 
\end{theorem}
\noindent An equivalent result is proved in \cite{BHR}; however,
here we propose a novel proof.
\noindent As a corollary of Theorem \ref{steiner} we obtain
a Steiner symmetry result for the minimizers in 
Theorem 3.9 in \cite{CC} and in Theorem \ref{cantrell2}. 
The biological meaning of Theorem \ref{steiner} is the
following: if the region $\Omega$ is Steiner symmetric, for
the population to survive the best environment is given when the 
favourable habitats are located far from the boundary $\partial
\Omega$ and arranged in a Steiner symmetrical fashion. As a
Steiner symmetric set is convex relative to a direction, the
overall favourable habitat cannot be made, at least in that direction, of disconnected pieces.
In other words, it should not be very fragmented.  This conclusion is a well-known biological fact (see for
example \cite{CC,CC91,BHR,SK}).\\
\noindent 
Problem \eqref{A} with positive bounded weight $m(x)$ also has 
a well known physical interpretation: it models a vibrating membrane
$\Omega$ with clamped boundary $\partial\Omega$ and mass
density $m(x)$; the physical meaning of $\lambda_1(m)$ is
the principal natural frequency of the membrane. Thus,
the minimization of $\lambda_1(m)$ is physically equivalent
to find the mass distribution of the membrane which gives the
lowest principal natural frequency. 
Among many papers that consider this interpretation of problem
\eqref{A}, we recall \cite{CGIKO,CML1,CML2}, where the
minimization
of $\lambda_1(m)$ is addressed in the case of a class of
weights which take only two positive values. \\
\noindent
This paper is structured as follows. In Section \ref{preliminaries} we describe
the eigenvalue 
problem \eqref{A} in some details and summarize some known
results about the minimization of $\lambda_1(m)$. Moreover,
we recall the definition of class of rearrangements of a 
measurable function and some related properties we will need
in the sequel.
In Section \ref{somme} we show a new formula involving
the classes of rearrangements, which we will use in Section
\ref{weights} in order to prove Theorem \ref{cantrell2}.
Finally, Section \ref{secsteiner} contains the proof of
Theorem \ref{steiner}.

\section{Preliminaries}\label{preliminaries}

We denote by $|E|$ the measure of an arbitrary Lebesgue measurable set  $E \subseteq\mathbb{R}^N$ and, when $N=1$, we also write $|E|
_1$.
We identify two measurable sets $E,F$ that are equal up to a
nullset, i.e. if $|E\smallsetminus F\cup F\smallsetminus E|=0$ and similarly for measurable
functions equal almost everywhere.

\subsection{The weighted eigenvalue problem}
\noindent Let $\Omega\subset\mathbb{R}^N$ be a bounded smooth domain and 
$m\in L^\infty(\Omega)$.
By $H_0^1(\Omega)$ and $W^{2,2}(\Omega)$ we denote the usual Sobolev spaces; we use the norm $\|u\|_{H_0^1(\Omega)}=\int_\Omega|\nabla u|^2\,dx$ (see \cite{GT}).
Problem \eqref{A} is considered in weak form:
$u\in H_0^1(\Omega)$ is a weak solution of \eqref{A} if
\begin{equation}\label{D}\int_\Omega\nabla u\cdot \nabla \varphi\,dx=\lambda\int_\Omega 
m u\varphi\,dx\quad\forall \varphi\in C^\infty_0(\Omega).
\end{equation}
A nontrivial solution of \eqref{A} is called an eigenfunction associated to the eigenvalue
$\lambda$. It is easy to check that zero is not an eigenvalue of problem \eqref{A}. 
The following proposition is a consequence of a result contained in \cite{dF}.
\begin{proposition}\label{segnom} With the notation above we have:\\
i)  if $|\{x\in \Omega:m(x) >0\}|> 0$, then there is a divergent sequence $\{\lambda_k(m)\}
_{k=1}^\infty$ of positive eigenvalues of problem \eqref{A};\\
ii)  if $|\{x\in \Omega: m(x) <0\}|> 0$, then there is a divergent sequence $\{\lambda_{-k}(m)\}_{k=1}^\infty$ of negative eigenvalues 
of problem \eqref{A}.
\end{proposition}
\noindent The smallest positive eigenvalue 
$\lambda_1(m)$, which we call the {\it principal positive eigenvalue},
 is simple and any associated eigenfunction is one-signed in $\Omega$ (see \cite [Theorem 
 1.13]{dF}), and similarly for $\lambda_{-1}(m)$.
\noindent In general, the eigenvalues of problem \eqref{A} form two monotone sequences
\begin{linenomath}
$$ 0< \lambda_1(m)< \lambda_2(m)\leq \ldots \leq \lambda_k(m)\leq\ldots$$
\end{linenomath}
and 
\begin{linenomath}$$\ldots\leq \lambda_{-k}(m)\leq \ldots\leq\lambda_{-2}(m)< \lambda_{-1}(m)<0,$$
 \end{linenomath}
where every eigenvalue is repeated with its multiplicity. 

\noindent Moreover, by classical regularity results, any eigenfunction $u$ related to $\lambda_1(m)$ 
belongs to $H^1_0(\Omega)\cap W^{2,2}(\Omega)\cap C^{1,\beta}(\overline{\Omega})$ for every $\beta\in(0,1)$ (see \cite{GT}). 
The principal positive eigenvalue $\lambda_1(m)$ has a variational characterization also known as the {\it Courant-Fischer Principle}

\begin{equation}\label{F}\frac{1}{\lambda_1(m)}\,= 
\max_{\genfrac{}{}{0pt}{3}{u\in H_0^1(\Omega)}{u\neq 0}}
\cfrac{\int_\Omega m u^2 \, dx}{\int_\Omega |\nabla u|^2\, dx }\,.
\end{equation}
\noindent
Furthermore, each maximizer of \eqref{F} is an eigenfunction associated to $\lambda_1(m)$ (see Proposition 1.10 and the proof of Lemma 1.1 in \cite{dF}).\\
As described in Section \ref{intro}, we consider the 
minimization of $\lambda_1(m)$ with respect to $m$, when $m$
belongs to a fixed class of bounded functions. Let us introduce in details
the classes we are interested in.

\subsection{Classes of rearrangements}\label{rearra}

Here we briefly recall some basic definitions and properties
about the rearrangements of measurable functions. For a
systematic and thorough treatment
of this subject we refer the reader to \cite{D}.

\begin{definition}\label{def1}
Let $E\subset\mathbb{R}^N$ be a set of finite measure. Two measurable functions $f,g:E \to \mathbb{R}$ are called \emph{equimeasurable}
or \emph{rearrangements of one another} if 
\begin{linenomath}
$$|\{x\in E: f(x)>t\}|=|\{x\in E: g(x)>t\}|\quad\forall\, t\in \mathbb{R}.$$
\end{linenomath}
\end{definition}

\noindent Equimeasurability of $f$ and $g$ is denoted by
$f\sim g$. Note that $\sim $ is an equivalence relation on the set 
of all measurable functions on $E$. Equimeasurable functions
share global extrema and integrals as it is precisely stated  
in the following proposition.

\begin{proposition}\label{rospo}
Let $E\subset\mathbb{R}^N$ be a set of finite measure and  $f,g:E \to \mathbb{R}$ two measurable functions. If $f\sim g$, then\\
i) $\text{esssup } f=\text{esssup } g$ and  $\text{essinf }  f=\text{essinf } g$;\\
ii) $f\in L^1(E)$ if and only if $g\in L^1(E)$, in which case $\int_Ef\,dx=\int_E g\,dx$;\\
iii) if $F:\mathbb{R}\to\mathbb{R}$ is a Borel measurable function, then $F\circ f\sim F\circ g$.
\end{proposition}
\noindent 
For a proof see, for example, \cite[Proposition 3.3]{D}. 
\begin{definition}\label{def2}
Let $E\subset\mathbb{R}^N$ be a set of finite measure. For every
measurable function $f:E\to\mathbb{R}$, the function 
$f^*:(0,|E|)\to\mathbb{R}$ defined by
\begin{linenomath}
$$f^*(s)=\sup\{t\in\mathbb{R}: |\{x\in E: f(x)>t\}|>s\}$$
\end{linenomath}
is called the \emph{decreasing rearrangement of $f$}.
\end{definition}
\noindent
It can be proved that $|\{x\in E: f(x)>t\}|=|\{x\in (0,|E|): f^*(x)>t\}|$ for all 
$t\in\mathbb{R}$ (see \cite[Theorem 5.2]{D}).
\begin{proposition} \label{furbi}
Let $f$ be as in Definition \ref{def2}. Then\\
i) $f^*$ is decreasing, right continuous and 
\begin{linenomath} $$\lim_{s\to 0}f^*(s)=\text{esssup } f\quad\text{and}\quad\lim_{s\to|E|}f^*(s)=\text{essinf } f;$$
\end{linenomath}
ii) $f\in L^1(E)$ if and only if $f^*\in L^1(0,|E|)$, in which case $\int_Ef\,dx=\int_0^{|E|} f^*\,ds$;\\
iii) if $F:\mathbb{R}\to\mathbb{R}$ is a Borel measurable function, then $F\circ f\in L^1(E)$ if and only if $F\circ f^*\in L^1(0,|E|)$, in which case $\int_EF\circ f\,dx=\int_0^{|E|} F\circ f^*\,ds$;\\
iv) for any measurable function $g$ on $E$ we have $g\sim f$ if and only if $g^*=f^*$.
\end{proposition}
\noindent
The proof easily follows from Definition \ref{def2} and Proposition 3.3 and 5.3 in \cite{D}.

\begin{definition}\label{prec1}
Let $E\subset\mathbb{R}^N$ be a set of finite measure and 
$f,g\in L^1(E)$. We write $g\prec f$ if
\begin{linenomath}
$$\int_0^t g^*\,ds\leq \int_0^t f^*\,ds\quad\forall\;
0\leq t\leq|E|\quad\quad
{and}\quad\quad\int_0^{|E|} g^*\,ds= \int_0^{|E|} f^*\,ds.$$
\end{linenomath}
\end{definition}

\begin{proposition}\label{prec2}
Let $E\subset\mathbb{R}^N$ be a set of finite measure and $f,g\in L^1(E)$. Then\\
i) $f\sim g$ if and only if $f\prec g$ and $g\prec f$;\\
ii) if $\alpha\leq f\leq\beta$ a.e. in $E$ with $\alpha,\beta\in\mathbb{R}$ and $g\prec f$,
then $\alpha\leq g\leq\beta$ a.e. in $E$;\\
iii) the constant function $c=\frac{1}{|E|}\int_E f\,dx$ precedes $f$, i.e.
$c\prec f$.\\
\end{proposition}

\noindent Property i) follows from Proposition \ref{furbi} and Definition \ref{prec1}; the
proof of ii) and iii) can be found in  \cite[Lemma 8.2]{D}.\\

\begin{definition}
Let $E\subset\mathbb{R}^N$ be a set of finite measure and $f:E\to\mathbb{R}$ a measurable function. We call the set
\begin{linenomath}
$$\mathcal{G}(f)=\{g:E\to\mathbb{R}: g \text{ is measurable and } g\sim f \}$$
\end{linenomath}
the \emph{class of rearrangements of $f$} or the \emph{set of rearrangements of $f$}.
\end{definition}
\noindent 
Note that, for $1\leq p\leq\infty$, if $f\in L^p(E)$ then
$\mathcal{G}(f)$ is contained in $L^p(E)$.
In this paper, we will consider rearrangements of bounded
functions. If $f_0\in L^\infty(E)$, we denote by $
\overline{\mathcal{G}(f_0)}$ the weak* closure of
$\mathcal{G}(f_0)$ in $L^\infty(E)$; 
$\overline{\mathcal{G}(f_0)}$ can be characterized by
the following property.

\begin{proposition}\label{prec}
Let $f_0\in L^\infty(E)$. Then
$\overline{\mathcal{G}(f_0)}=\{f\in L^\infty(E): f\prec f_0\}
$.
\end{proposition}

\noindent The claim follows from Theorem 22.13 and Theorem 22.2 in
\cite{D}.

\begin{proposition}[Hardy-Littlewood inequality]\label{s2}
Let $E\subset\mathbb{R}^N$ be a set of finite measure and
$f, g:E \to \mathbb{R}$ two measurable functions  such that 
$|f|^*|g|^*\in L^1(0,|E|)$. Then $fg\in L^1(E)$ and
\begin{equation}\label{HLr}
\int_E f(x)g(x)\,dx\leq \int_0^{|E|}f^*
g^*\,ds.
\end{equation}
Moreover, there exists $\tilde{g}\sim g$ such that, inserted
in \eqref{HLr} in place of $g$, gives the equality sign in
\eqref{HLr}.
\end{proposition}
\noindent The above proposition follows immediately from
Theorem 10.1 and Theorem 11.1 in \cite{D}.
\begin{corollary}\label{corhl}
 Let $E\subset\mathbb{R^N}$ be a set of finite measure,
$f:E \to \mathbb{R}$ a measurable function and $t\in(0,|E|)$  such that $|f|^*\in L^1(0,t)$. Then, $f\in L^1(A)$ for every
measurable subset $A$ of $E$ of measure $t$ and the identity
\begin{equation}\label{supA}
\int_0^tf^*ds=\sup_{\genfrac{}{}{0pt}{2}{A\subseteq E,}{|A|=t}}\int_Af\,dx
\end{equation}
holds.
\end{corollary}
\begin{proof}
In Proposition \ref{s2}, choose $g=\chi_A$ for any measurable subset $A$ of $E$ of measure $t$. By \eqref{HLr} we find
\begin{equation}\label{corr}
\int_Af\,dx=\int_E f\chi_A\,dx\leq \int_0^{|E|}f^*
(\chi_A)^*\,ds=\int_0^tf^*\,ds.
\end{equation}
The subset $A$ being arbitrary in $E$, we can write
\begin{linenomath}
$$\sup_{\genfrac{}{}{0pt}{2}{A\subseteq E,}{|A|=t}} \int_Af\,dx\leq \int_0^tf^*
\,ds.$$
\end{linenomath}
Now, fix $A$ in \eqref{corr} and apply the second part 
of Proposition \ref{s2}. There exists $\tilde g\sim\chi_A$
such that
\begin{linenomath}
$$\int_E f\tilde g\,dx= \int_0^{|E|}f^*
\tilde g^*\,ds.$$
\end{linenomath}
By Definition \ref{def1}, it is not difficult to show that
$\tilde g\sim\chi_{A}$ implies $\tilde g=\chi_{\tilde A}$
with $\tilde A$ of measure $t$. Thus, we find
\begin{linenomath}
$$\int_{\tilde A} f\,dx= \int_0^tf^*\,ds,$$
\end{linenomath}
which concludes the proof. 
\end{proof}

\noindent Finally, we state Theorem 18.10 in \cite{D} for our case.
\begin{proposition}\label{dayprec}
Let $E\subset\mathbb{R}^N$ be a set of finite measure and
$f_1,f_2,g\in L^1(E)$. If  $g\prec f_1+f_2$, then there exist $g_1,g_2\in L^1(E)$
such that $g=g_1+g_2$ with $g_1\prec f_1$ and $g_2\prec f_2$.
\end{proposition}

\subsection{The minimization of $\lambda_1(m)$}

As mentioned in Section \ref{intro}, many authors investigated
the minimization of $\lambda_1(m)$ as a function of the weight
$m$ in connection with mathematical ecology. Such a minimization is examined over a class of functions 
which reflects some biological constraints.\\
Here we are mainly concerned with the papers \cite{CC} and
\cite{CCP}. In particular, in \cite{CCP} the authors consider
the optimization of $\lambda_1(m)$ as $m$ varies in a class of rearrangements. Their results about existence and characterization of a minimizer $\check{m}$ of $\lambda_1(m)$
can be summarized and adapted to our treatment as follows.

\begin{proposition}\label{CCP}
Let $\Omega\subset\mathbb{R}^N$ be a bounded smooth domain and
$m_0\in L^\infty(\Omega)$ such that $|\{x\in\Omega:m_0(x)>0\}|>0$. Then, there exists a solution $\check{m}\in\mathcal{G}(m_0)$ of the minimization problem
\begin{equation}\label{M}
\inf\{\lambda_1({m}):{m}\in\overline{\mathcal{G}({m}_0)}\}.
\end{equation}
Furthermore, if $\check m\in \mathcal{G}({m}_0)$ is an arbitrary solution of \eqref{M}
and $u_{\check{m}}$ denotes the unique positive eigenfunction normalized by $\|u_{\check m}\|_{H_0^1(\Omega)}= 1$ of \eqref{A} with $m=\check{m}$ and $\lambda=
\lambda_1(\check m)$, then there exists an increasing function $\psi$ such that $\check{m}=\psi(u_{\check{m}})$
a.e. in $\Omega$.
\end{proposition} 
\begin{proof}
By Proposition 3.6 in \cite{ACF}, $\overline{\mathcal{G}({m}_0)}$ is weakly* compact and metrizable in the weak* topology.
Therefore $\mathcal{G}({m}_0)$ is dense in 
$\overline{\mathcal{G}({m}_0)}$. By Proposition 2.1 of 
\cite{CCP}, there is a solution of the problem \eqref{M}.
By Theorem 2.1 of \cite{CCP}, there is a minimizer 
$\check m \in \mathcal{G}({m}_0)$ and, for any
such minimizer, there exists an increasing function
$\psi$ such that $\check{m}=\psi(u_{\check{m}})$
a.e. in $\Omega$.
\end{proof}

\begin{remark}\label{prossimo}
Reasoning as in the proof of Theorem 1.1 in \cite{ACF}, it can be shown that all the
minimizers of problem \eqref{M} belong to $\mathcal{G}({m}_0)$.  As already noted
after Theorem \ref{cantrell2}, we will include this result in a future paper.
\end{remark}

\noindent As secondary result, our work shows how Theorem 3.9 in \cite{CC} can be
regarded as a corollary of Proposition \ref{CCP}.

\section{Sums of closures of classes of rearrangements}\label{somme}

In this section we prove a new property about the classes of
rearrangements of measurable functions.
Let $E\subset\mathbb{R}^N$ be a set of finite measure.

\begin{definition}
Let $V$ be a real vector space, $A_1, \ldots, A_n\subseteq V$ and $\alpha\in \mathbb{R}$. We define
\begin{linenomath}$$\sum_{i=1}^n A_i= \left\{ v\in V: v=\sum_{i=1}^n a_i, a_i\in A_i, i=1, \ldots, n\right\}$$
 \end{linenomath}
and
\begin{linenomath}$$\alpha A=\{v\in V: v=\alpha a, a\in A\}.$$
 \end{linenomath}
\end{definition}

\noindent As stated in the following theorem, 
when $A_i=\overline{\mathcal{G}(f_i)}$ for some bounded
functions $f_i$, $i=1, \ldots,n$, their sum is equal to
$\overline{\mathcal{G}(\textstyle\sum_{i=1}^{n}f_i)}$
provided every pair of functions $f_i, f_j$ realizes the equality in \eqref{HLr}. 

\begin{theorem}\label{sommacla}
Let $f_1,\ldots,f_n\in L^\infty(E)$ be such that $\int_E f_i f_j\,dx=\int_0^{|E|}f_i^*f_j^*\,ds$ for all $i,j=1,\ldots,n$, $i\neq j$. Then
\begin{linenomath}
\begin{equation*}
\sum_{i=1}^{n}\overline{\mathcal{G}(f_i)}=\overline{\mathcal{G}(\textstyle\sum_{i=1}^{n}f_i)}.\end{equation*}
\end{linenomath}
\end{theorem}

\begin{proof}
First, we show that $\overline{\mathcal{G}(\textstyle\sum_{i=1}^{n}f_i)}\subseteq\sum_{i=1}^{n}\overline{\mathcal{G}(f_i)}$. Assume $h\in\overline{\mathcal{G}(\textstyle\sum_{i=1}^{n}f_i)}$; by Proposition \ref{prec}, $h
\prec \textstyle\sum_{i=1}^{n}f_i$.
Using Proposition \ref{dayprec} $n-1$ times, there exists 
a set of $n$ integrable functions $h_1,\ldots, h_n$ 
such that $h=\sum_{i=1}^n h_i$ and $h_i\prec f_i$, i.e. $h_i\in\overline{\mathcal{G}(f_i)}$ for all $i=1,\ldots,n$. 
Therefore $h\in\sum_{i=1}^{n}\overline{\mathcal{G}(f_i)}$.\\
 Second, we prove the opposite inclusion $\sum_{i=1}^{n}\overline{\mathcal{G}(f_i)}\subseteq\overline{\mathcal{G}(\textstyle\sum_{i=1}^{n}f_i)}$.  We claim that 
 \begin{linenomath}
\begin{equation}\label{somma}\left(\sum_{i=1}^n f_i\right)^*=\sum_{i=1}^n f_i^*.\end{equation}
\end{linenomath}
Putting $f=\sum_{i=1}^n f_i$,
 using iii) of Proposition \ref{furbi} (with $F$ equal to the
 square function), Proposition \ref{s2} and the
 assumption  $\int_E f_if_j\,dx=\int_0^{|E|}f_i^*f_j^*\,ds$ for
 $i,j=1,\ldots,n$, $i\neq j$, we have
\begin{linenomath}
\begin{equation*}
\begin{split}
\int_0^{|E|}\left(f^*-\sum_{i=1}^n f_i^*\right)^2\,ds
&=\int_0^{|E|}(f^*)^2\,ds+\sum_{i=1}^{n}\int_0^{|E|}(f_i^*)^2\,ds\\
&-2\sum_{i=1}^{n}\int_0^{|E|}f^*f_i^*\,ds+2\sum_{\genfrac{}{}{0pt}{2}{i,j=1}{i\neq j}}
^{n}\int_0^{|E|}f_i^*f_j^*\,ds\\
&\leq\int_Ef^2\,dx+\sum_{i=1}^{n}\int_Ef_i^2\,dx\\
&-2\sum_{i=1}^{n}\int_Eff_i\,dx+2\sum_{\genfrac{}{}{0pt}{2}{i,j=1}{i\neq j}}^{n}
\int_Ef_if_j\,dx\\
&=\int_E\left(f-\sum_{i=1}^n f_i\right)^2\,dx=0
\end{split}
\end{equation*}
\end{linenomath}
which proves the claim.\\
Now, let $h=\sum_{i=1}^n h_i\in\sum_{i=1}^{n}\overline{\mathcal{G}(f_i)}$ with $h_i\prec f_i$ for $i=1,\ldots,n$.
For each $t\in(0,|E|)$, using \eqref{supA} and  \eqref{somma} we have
\begin{linenomath}\begin{equation*}
\begin{split}
\int_0^th^*\,ds&=\int_0^t\left(\sum_{i=1}^n h_i\right)^*ds=\sup_{\genfrac{}{}{0pt}{2}{A\subseteq E}{|A|=t}}\int_A\sum_{i=1}^n h_i\,dx
\leq\sum_{i=1}^n\sup_{\genfrac{}{}{0pt}{2}{A\subseteq E}{|A|=t}}\int_A h_i\,dx\\
&=\sum_{i=1}^n\int_0^th_i^*\,ds\leq\sum_{i=1}^n\int_0^t f_i^*\,ds=\int_0^t\left(\sum_{i=i}^nf_i\right)^*ds;
\end{split}
\end{equation*}
\end{linenomath}
moreover, using ii) of Proposition \ref{furbi} and 
$h_i\prec f_i$ for $i=1,\ldots,n$, we find
\begin{linenomath}
\begin{equation*}
\begin{split}
\int_0^{|E|}h^*\,ds&=\int_0^{|E|}\left(\sum_{i=1}^n h_i\right)^*ds=\int_E\sum_{i=1}^n h_i\,dx=\sum_{i=1}^n\int_0^{|E|}h_i^*\,ds\\
&=\sum_{i=1}^n\int_0^{|E|}f_i^*\,ds
=\int_E\sum_{i=1}^n f_i\,dx=\int_0^{|E|}\left(\sum_{i=1}^n f_i\right)^*\,ds.
\end{split}
\end{equation*}
\end{linenomath}
Therefore $h\prec\sum_{i=1}^n f_i$, i.e. $h\in\overline{\mathcal{G}(\textstyle\sum_{i=1}^{n}f_i)}$.
This completes the proof.
\end{proof}

\noindent 
Even if the functions $f_i$'s do not realize the equality in \eqref{HLr}, it is still possible to write
$\sum_{i=1}^{n}\overline{\mathcal{G}(f_i)}$ as closure of the class of rearrangements of a suitable function.
Precisely, the following theorem holds.

\begin{theorem}\label{scambio}
Let $g_1,\ldots,g_n\in L^\infty(E)$. Then, there exist
 $f_1,\ldots,f_n$ such that $f_i\sim g_i$  and $\int_E f_i f_j\,dx=\int_0^{|E|}f_i^*f_j^*\,ds$ for all $i,j=1,\ldots,n$, $i\neq j$. 
\end{theorem}
\begin{proof}
We put $f_1=g_1$ and by Proposition \ref{s2} we can choose $f_2\sim g_2$ such that $\int_E f_1 f_2\,dx=\int_0^{|E|}f_1^*f_2^*\,ds$. As in the previous proof, we have $(f_1+f_2)^*=f_1^*+f_2^*$.
Similarly, we can select $f_3\sim g_3$ so as to
\begin{linenomath}
$$\int_E (f_1+f_2) f_3\,dx=\int_0^{|E|}(f_1+f_2)^*f_3^*\,ds$$
\end{linenomath}
or, equivalently,
\begin{linenomath}
$$\int_E f_1f_3\,dx+\int_E f_2f_3\,dx=\int_0^{|E|}f_1^*f_3^*\,ds+\int_0^{|E|}f_2^*f_3^*\,ds.$$
 \end{linenomath}
By the Hardy-Littlewood inequality \eqref{HLr} applied to the pairs $f_1,f_3$ and $f_2,f_3$ the last equation leads to
\begin{linenomath}$$\int_E f_1 f_3\,dx=\int_0^{|E|}f_1^*f_3^*\,ds\quad\text{and}\quad\int_E f_2 f_3\,dx=\int_0^{|E|}f_2^*f_3^*\,ds.$$
\end{linenomath}
Once we have found $f_1,\ldots,f_k$ such that $f_i\sim g_i$ and $\int_E f_i f_j\,dx=\int_0^{|E|}f_i^*f_j^*\,ds$ for all $i,j=1,\ldots,k$, $i\neq j$,
we can choose $f_{k+1}\sim g_{k+1}$ such that
\begin{linenomath}
$$\int_E (f_1+\dots+f_k) f_{k+1}\,dx=\int_0^{|E|}(f_1+\dots+f_k)^*f_{k+1}^*\,ds.$$
\end{linenomath}
Using again \eqref{HLr} we conclude $\int_E f_i f_j\,dx=\int_0^{|E|}f_i^*f_j^*\,ds$ for all $i,j=1,\ldots,k+1$, $i\neq j$. After a finite number of steps we find all the functions $f_1,\ldots,f_n$ of the statement.
\end{proof}

\begin{remark}\label{oss1}
As noted above, from Theorem \ref{sommacla} and Theorem \ref{scambio} it
follows that for arbitrary bounded functions $g_1,\ldots,g_n$, there exist
$f_1,\ldots,f_n$, with $f_1\sim g_1,\ldots,f_n\sim g_n$, such that
\begin{linenomath}
$$\sum_{i=1}^{n}\overline{\mathcal{G}(g_i)}=\overline{\mathcal{G}(\textstyle\sum_{i=1}^{n}f_i)}.$$
\end{linenomath}
\end{remark}
\noindent We also note that the product of the closure of a class of rearrangements by a real number is again a closure of a class of rearrangements.

\begin{theorem}
Let $f$ be a bounded measurable function. Then
\begin{equation}\label{omo}
\alpha\overline{\mathcal{G}(f)}=\overline{\mathcal{G}(\alpha f)}
\end{equation}
 for all $\alpha\in\mathbb{R}$.
\end{theorem}
\begin{proof}
 It is an immediate consequence of (ii) in \cite[Lemma 8.2]{D}.
\end{proof}

\section{Classes of weights $m(x)$ and proof of Theorem \ref{cantrell2}}
\label{weights}

In this section $\Omega\subset\mathbb{R}^N$ denotes a 
measurable set of finite measure.\\
In \cite{CC} the authors considered the class of weights \begin{equation}\label{Mcos}
\begin{split}
 \mathcal{M}=\, &\Bigg\{m(x)\in L^\infty(\Omega):
 -m_2\leq m(x)\leq m_1\text{ a.e. in }\Omega,\int_\Omega m\,dx=m_3 \\ &\text{ and }
 m(x)>0 \text{ on a set of positive measure}\}
\end{split} 
\end{equation}
where $m_1,m_2$ and $m_3$ are constants with $m_1$ and $m_2$ positive. They 
prove that there exists a solution of the problem 
\begin{equation}\label{optimizationCC}
\inf\{\lambda_1(m):m\in\mathcal{M}\}
\end{equation}
of the form $m_1\chi_E-m_2\chi_{\Omega\smallsetminus E}$,
where $E$ is a measurable subset of $\Omega$.
Therefore, problem \eqref{optimizationCC} can be recast as a shape optimization problem on the set $\mathcal{N}=\{m_1\chi_E-m_2\chi_{\Omega\smallsetminus E}: E\subseteq\Omega \text{ is measurable, with } m_1|E|-m_2|\Omega\smallsetminus E|=m_3\}$.
In fact, as it is shown in the following lemmas,  the set 
$\mathcal{N}$ coincides with a class of rearrangements and $\mathcal{M}$ can be written by means of the weak* closure of $\mathcal{N}$. 
Therefore, problem \eqref{optimizationCC}
can be seen as a particular case of problem \eqref{M}.

\begin{lemma}\label{L}
Let $\Omega\subset\mathbb{R}^N$ be a measurable set of finite measure and $m_1,m_2,m_3$ constants such that  $-m_2|\Omega|\leq m_3\leq m_1|\Omega|$. Then, the set of functions $\mathcal{N}=\{m_1\chi_E-m_2\chi_{\Omega\smallsetminus E}: E\subseteq\Omega \text{ is measurable, with } m_1|E|-m_2|\Omega\smallsetminus E|=m_3\}$ coincides with the  class of rearrangements
$\mathcal{G}(m_0)$, where $m_0$ is an arbitrary fixed element of $\mathcal{N}$.
\end{lemma}
\begin{proof}
If an equality sign holds in $-m_2|\Omega|\leq m_3\leq m_1|\Omega|$, then the set $
\mathcal{N}$ reduces to a singleton containing a constant function. In this case the 
statement is trivially true, hence in the rest of the proof we assume  $-m_2|\Omega|<
m_3< m_1|\Omega|$. Clearly, for any element of $\mathcal{N}$, the set $E$ has measure $e=(m_2|\Omega|+m_3)/(m_1+m_2)$. We recall that a class of rearrangements  is an equivalence class with respect to the equimeasurability relation among measurable
functions. First, we show that all the elements of $\mathcal{N}$ are  equimeasurable. This follows immediately from the identity 
\begin{equation}\label{CA}
 |\{x\in \Omega: f(x)>t\}|=
\begin{cases}
|\Omega|&\text{ if } t<-m_2,\\
e&\text{ if } -m_2\leq t<m_1,\\
0&\text{ if } t\geq m_1
\end{cases} 
\end{equation}
for each $f\in\mathcal{N}$.
Now, let $f$ be a measurable function which satisfies \eqref{CA}. We will show that $f\in\mathcal{N}$ and this will complete the proof. For abbreviation, by $\{f>t\}$ we mean $\{x\in \Omega:f(x)>t\}$ and similarly for $\{f=t\}$ and $\{f\geq t\}$. Applying elementary measure
theory to the identity $\{f\geq t\}=\cap_{k=1}^\infty\{f> t-1/k\}$
for $t=m_1,-m_2$ and using \eqref{CA} we find $|\{f\geq m_1\}|=e$
and $|\{f\geq -m_2\}|=|\Omega|$. Finally, from $|\{f=t\}|=|\{f\geq t\}|-|\{f>t\}|$ and  \eqref{CA} again for $t=m_1,-m_2$, we get $|\{f= m_1\}|=e$ and $|\{f= -m_2\}|=|\Omega|-e$, which imply $f\in\mathcal{N}$.
\end{proof}
\noindent In other words, the class of rearrangements
$\mathcal{N}=\mathcal{G}(m_0)$ consists of ``bang-bang'' functions. As the
following lemma shows, it turns out that the class $\mathcal{
M}$ considered in \cite[Theorem 3.9]{CC} coincides with
the elements of the weak* closure of $\mathcal{G}(m_0)$ 
in $L^\infty(\Omega)$ which are positive in a subset
of $\Omega$ of positive measure. 

\begin{lemma}\label{cantrell}
Let $\Omega\subset\mathbb{R}^N$ be a measurable set of finite measure and $\mathcal{F}=\{m(x)\in L^\infty(\Omega): -m_2\leq m(x)\leq m_1\text{
a.e. in }\Omega\text{ and }\int_\Omega m\,dx=m_3\}$, with $m_1,m_2,m_3$ constants 
and $-m_2|\Omega|\leq m_3\leq
m_1|\Omega|$. Moreover, let $e$ be a nonnegative constant such that $(m_1+m_2)e
=m_2|\Omega|+m_3$. Then $\mathcal{F}
=\overline{\mathcal{G}(m_0)}$, where $m_0$ is an arbitrary function such that $
m_0^*=m_1\chi_{(0,e)}-m_2\chi_{(e,|\Omega|)}$ 
or, equivalently, is of the form
$m_0=m_1\chi_{E_0}-m_2\chi_{\Omega\setminus E_0}$, where $E_0
\subseteq\Omega$ has measure $e$. 
\end{lemma}
\begin{proof}
Let $m\in\overline{\mathcal{G}(m_0)}$. By Proposition \ref{prec} and ii) of Proposition
\ref{prec2}, we have $-m_2\leq m(x)\leq m_1$ a.e. in $\Omega$.
Moreover, by ii) of Proposition \ref{furbi}, Proposition \ref{prec} and Definition
\ref{prec1} we find $\int_\Omega m\,dx=\int_0^{|\Omega|} m^*\,ds=\int_0^{|\Omega|} m_0^*\,ds=\int_\Omega m_0\,dx=m_3$.
Therefore, $m\in\mathcal{F}$ and then $\overline{\mathcal{G}(m_0)}
\subseteq\mathcal{F}$.\\
Now assume $m\in \mathcal{F}$. Using i) and ii) of Proposition \ref{furbi} we obtain 
$-m_2\leq m^*(x)\leq m_1$ a.e. in $\Omega$ and $\int_0^{|\Omega|}
m^*\,ds=\int_\Omega m\,dx=m_3=\int_\Omega m_0\,dx=\int_0^{
|\Omega|}m_0^*\,ds$.

%Proposition \ref{prec} and Definition \ref{prec1} we obtain  and .

\noindent Fix $t\leq e$; we have
\begin{linenomath}\begin{equation*}
\int_0^tm^*\,ds\leq\int_0^t m_1\,ds=\int_0^tm_0^*\,ds.
\end{equation*}
\end{linenomath}
If, instead $t\geq e$, we find
\begin{linenomath}\begin{equation*}
\begin{split}
\int_0^tm^*\,ds&=\int_0^{|\Omega|}m^*\,ds-\int_t^{|\Omega|}m^*\,ds\leq\int_0^{|\Omega|}m^*\,ds+\int_t^{|\Omega|}m_2\,ds\\
&= \int_0^{|\Omega|}m_0^*\,ds-\int_t^{|\Omega|}m_0^*\,ds=\int_0^tm_0^*\,ds.
\end{split}
\end{equation*}\end{linenomath}
Then, by Proposition \ref{prec} and Definition \ref{prec1} we find $m\in\overline{\mathcal{G}(m_0)}$ and then
$\mathcal{F}\subseteq\overline{\mathcal{G}(m_0)}$.
The proof is completed.
\end{proof}
\noindent
The following remark contains an alternative and shorter proof of 
Theorem 3.9 in \cite{CC}.

\begin{remark}
By using Lemmas \ref{L} and \ref{cantrell},  
Theorem 3.9 in \cite{CC} follows almost immediately
from Proposition \ref{CCP}. Indeed, let $\mathcal{M}$
be as in Theorem 3.9 in \cite{CC} (i.e. the set \eqref{Mcos})
and let $L^+=\{m(x)\in L^\infty(\Omega):|\{m(x)>0\}|>0\}$.
By Lemma \ref{cantrell} we have $\mathcal{M}=\overline{\mathcal{G}(m_0)}\cap L^+$, where
$m_0=m_1\chi_{E_0}-m_2\chi_{\Omega\setminus E_0}$ and
$E_0$ is an arbitrary measurable subset of $\Omega$
of measure $(m_2|\Omega|+m_3)/(m_1+m_2)$. Furthermore, $m_0$ belongs to the set 
$\mathcal{N}$ of Lemma 1; thus, by the same lemma, $\mathcal{G}(m_0)=\mathcal{N}$. By Proposition \ref{CCP} (being $m_0\in L^+$), there exists 
a minimizer $\check m\in \mathcal{G}(m_0)$ of the problem
\eqref{M}. $\mathcal{G}(m_0)=\mathcal{N}$ implies that
$\check m=m_1\chi_{E}-m_2\chi_{\Omega\setminus E}$ for some
subset $E$ of $\Omega$ such that $|E|=(m_2|\Omega|+m_3)/(m_1+m_2)$. Moreover, $\check m\in L^+$ and then $\check m\in
\mathcal{M}$. Finally, $\check m$ being a minimizer in $\overline{\mathcal{G}(m_0)}$,
it minimizes $\lambda_1(m)$ in $\mathcal{M}$ as well. This prove Theorem 3.9 in
\cite{CC}.

 \end{remark}
\noindent
We now have all the necessary tools to prove Theorem
\ref{cantrell2}.

\begin{proof}[Proof of Theorem \ref{cantrell2}]
By Lemma \ref{cantrell} we have $\mathcal{F}_i=\overline{\mathcal{G}(f_i)}$,  where
$f_i=q_i\chi_{E_i}-p_i\chi_{\Omega\setminus E_i}$ with $E_i\subseteq\Omega$ has measure
\begin{equation}\label{costanti}
e_i=\frac{p_i|\Omega|+l_i}{p_i+q_i},\quad
i=1,2. 
\end{equation}
By Theorem \ref{scambio}, we can replace $f_1$ and $f_2$ (however, we use the same 
symbols)  without changing their class of rearrangements in such a way that $
\int_\Omega f_1 f_2\,dx=\int_0^{|\Omega|}f_1^*f_2^*\,ds$; 
moreover, by Lemma \ref{L}, $f_1$ and $f_2$ are still of the form written above.
By Theorem \ref{sommacla}, we have 
$\overline{\mathcal{G}(f_1)}+\overline{\mathcal{G}(f_2)}=
\overline{\mathcal{G}(f_1+f_2)}$; thus, 
\begin{equation}\label{capL}
\mathcal{M}=(\mathcal{F}_1+\mathcal{F}_2)\cap L^+=
\left(\overline{\mathcal{G}(f_1)}+
\overline{\mathcal{G}(f_2)}\right)\cap L^+=\overline{\mathcal{G
}(f_1+f_2)}\cap L^+. 
\end{equation}
 Arguing as in the proof
of Theorem \ref{sommacla} we find $(f_1+f_2)^*=f_1^*+f_2^*$,
which, exploiting the form of $f_1$ and $f_2$, implies
\begin{equation}\label{f1f2}
(f_1+f_2)^*=(q_1+q_2)\chi_{(0,\gamma)}+r\chi_{(\gamma,\delta)}-(p_1+p_2)\chi_{(
\delta,|\Omega|)},
\end{equation}
where 
$\gamma=\min\{e_1,e_2\}$, $\delta=\max\{e_1,e_2\}$
and
\begin{linenomath}
\begin{equation*}
r=\begin{cases}
q_1-p_2 &\text{if }e_1>e_2,\\
0 & \text{if } e_1=e_2,\\
q_2-p_1 &\text{if } e_1<e_2.
\end{cases}  
\end{equation*}
\end{linenomath}
Equality \eqref{f1f2} implies that $m_0=f_1+f_2$ is a weight
of the form $(q_1+q_2)\chi_{E_0}+ r\chi_{G_0\cap(\Omega\smallsetminus E_0)}-(p_1+p_2)\chi_{\Omega\smallsetminus G_0}$ for two  
subsets $E_0,G_0$ of $\Omega$ such that $E_0\subseteq G_0$ and $|E_0|=\gamma, 
|G_0|=\delta$.
 By Proposition \ref{CCP} (being $m_0\in L^+$), there exists 
a minimizer $\check m\in \mathcal{G}(m_0)$ of the problem
\eqref{M}. Reasoning similarly as we did in the proof of Lemma
\ref{L}, we find that $\check m=(q_1+q_2)\chi_{E}+
r\chi_{G\cap(\Omega\smallsetminus E)}-(p_1+p_2)\chi_{\Omega\smallsetminus G}$ for
two subsets $E,G$ of $\Omega$ such that $E\subseteq G$ and  $|E|=\gamma, 
|G|=\delta$. Moreover, $\check m\in L^+$ and, by \eqref{capL}, $\check m\in
\mathcal{M}$. Clearly,  $\check m$ is a minimizer in $\mathcal{M}$.\\
\noindent
Let $\check{f}_i\in\mathcal{F}_i$, $i=1,2$, such that $\check{m}=\check{f}_1+\check{f}
_2$. Since the subsets $E, G\cap(\Omega\smallsetminus E), G$ form a partition of $\Omega$
and $-p_i\leq \check{f}_i\leq q_i$, $i=1,2$, a.e. in $\Omega$,  we obtain 
$$\check{f}_1=
q_1\chi_E+\check{f}_1\chi_{G\cap(\Omega\smallsetminus E)}-
p_1\chi_{\Omega\smallsetminus G}$$
and
$$\check{f}_2=
q_2\chi_E+\check{f}_2\chi_{G\cap(\Omega\smallsetminus E)}-
p_2\chi_{\Omega\smallsetminus G}.$$
If $e_1=e_2$ we have $|G\cap(\Omega\smallsetminus E|=\delta-\gamma=0$ and $E=G$,
which yield the statement in this case. Let us now assume $e_1\neq e_2$. Integrating 
$\check{f}_i$ over $\Omega$ and because $\check{f}_i\in\mathcal{F}_i$, $i=1,2$, we 
get 
\begin{linenomath}
$$
q_i\gamma+\int_{G\cap(\Omega\smallsetminus E)}\check{f}_i\,dx-p_i(|\Omega|-\delta)=
l_i,\quad i=1,2,
$$
\end{linenomath}
which, by \eqref{costanti}, gives
\begin{linenomath}$$
\int_{G\cap(\Omega\smallsetminus E)}\check{f}_i\,dx=q_i(e_i-\gamma)-p_i(\delta-e_i),\quad i=1,2,
$$\end{linenomath}
and then
\begin{linenomath}$$
\int_{G\cap(\Omega\smallsetminus E)}\check{f}_i\,dx=
\begin{cases}
 q_i(\delta-\gamma) \quad & \text{if } e_i=\delta,\\ 
 -p_i(\delta-\gamma)  & \text{if } e_i=\gamma,
\end{cases}
\quad i=1,2.
$$\end{linenomath}
Recalling that $|G\cap(\Omega\smallsetminus E|=\delta-\gamma$ and $-p_i\leq \check{f}
_i\leq q_i$, $i=1,2$, we conclude
\begin{linenomath}$$
\check{f}_i=
\begin{cases}
 q_i \quad & \text{if } e_i=\delta,\\ 
 -p_i  & \text{if } e_i=\gamma,
\end{cases}
\quad i=1,2.$$
\end{linenomath}
a.e. in $G\cap(\Omega\smallsetminus E)$.
If $e_1>e_2$ we find 
$\check{f}_1=
q_1\chi_E+q_1\chi_{G\cap(\Omega\smallsetminus E)}-
p_1\chi_{\Omega\smallsetminus G}=q_1\chi_G-p_1\chi_{\Omega\smallsetminus G}$
and $\check{f}_2=q_2\chi_E-p_2\chi_{G\cap(\Omega\smallsetminus E)}-
p_2\chi_{\Omega\smallsetminus G}=q_2\chi_E-p_2\chi_{\Omega\smallsetminus E}$
which, by interchanging the symbols $E$ and $G$, gives the claim of the statement
for $e_1>e_2$. The case $e_1<e_2$ can be similarly treated.
This concludes the proof. 
\end{proof}
\noindent
We note that the previous proof can be easily extended to the case of a
finite number of sets $\mathcal{F}_i$, $i=1,\ldots,n$.

\begin{remark}
Let us consider an example. Suppose we have to distributes two different resources
through an environment $\Omega$. The first resource $\mathcal{F}_1$ affects positively
the survival of the population (for example food) whereas the second $\mathcal{F}_2$
does it negatively (for example predators). Moreover we have two independent constraints
on the total amounts of $\mathcal{F}_1$ and $\mathcal{F}_2$. Choosing the parameters in
order to keep very simple our calculation, we write
\begin{linenomath}
$$\mathcal{F}_1=\left\{f_1(x)\in L^\infty(\Omega): 0\leq f_1(x)\leq 
1\text{ a.e. in }\Omega\text{ and }\int_\Omega f_1
\,dx=\frac{2|\Omega|}{3}\right\}$$
\end{linenomath}
and
\begin{linenomath}
$$\mathcal{F}_2=\left\{f_2(x)\in L^\infty(\Omega): -1\leq  f_2(x)\leq 
0\text{ a.e. in }\Omega\text{ and }\int_\Omega f_2\,dx=-\frac{|\Omega|}{2}\right\}.$$
\end{linenomath}
We find $e_1=2|\Omega|/3,\, e_2=|\Omega|/2$. Therefore $e_1>e_2$ and then an optimal
location of the two resources is given by the local grow rate $\check m_1=\chi_E-
\chi_{\Omega\smallsetminus G}$ where  $E,G$ are suitable subsets of $\Omega$ such 
that $G\subseteq E$ and $|E|=2|\Omega|/3$, $|G|=|\Omega|/2$. Note that there is a region
of $\Omega$ ($E\smallsetminus G$) which contains both resources, one at its maximum
density and the other at its minimum. In general, this fact occurs if and only if $e_1\neq
e_2$.\\
It is instructive to examine which is the optimal local grow rate if we replace the two
constraints on the resources by a unique condition. In this case the resources are described 
by a unique density function $f(x)$ belonging to the class
\begin{linenomath}
\begin{equation*}
 \mathcal{F}=\Bigg\{f(x)\in L^\infty(\Omega):
 -1\leq f(x)\leq 1\text{ a.e. in }\Omega, \int_\Omega f\,dx=\frac{|\Omega|}{6}
\Bigg\}.
\end{equation*}
\end{linenomath}
By using Theorem 3.9 in \cite{CC} one finds as the best local grow rate
$\check m_2=\chi_E-\chi_{\Omega\smallsetminus E}$ where $E$ is a subset of
$\Omega$ such that $|E|=e=7|\Omega|/12$. Using Remark \ref{prossimo}
it can be shown that  $\lambda_1(\check m_2)<\lambda_1(\check m_1)$, i.e. 
$\check m_2$ is a better local grow rate than $\check m_1$.\\
Summarizing, $\check m_1$ and $\check m_2$ are two different trade-offs: the former
gives the best arrangement of the resources when the amount of each of them is
fixed, the latter yields a better chance of survival of the population but it satisfies
a less stringent condition on the availability of the resources.
\end{remark}

\section{Steiner Symmetry}\label{secsteiner}

\noindent Following \cite{Bro}, we introduce the notion of Steiner symmetrization.\\
Let $l(x')= \{x=(x_1,x')\in \mathbb{R}^N: x_1\in \mathbb{R}\}$ for any fixed $x'\in \mathbb{R}^{N-1}$  and let  $T$ be the hyperplane $\{x=(x_1,x')\in \mathbb{R}^N: x_1=0\}$.

\begin{definition}%\label{steiner2}
Let $E\subset \mathbb{R}^N$ be a measurable set.\\
Then the set $
E^\sharp=\left\{x=(x_1,x')\in\mathbb{R}^N:2|x_1|<|E \cap
l(x')|_1, x' \in\mathbb{R}^{N-1}\right\}$
is said the \emph{Steiner symmetrization} of $E$ with respect to the hyperplane $T$ and $E$ is called \emph{Steiner symmetric} if $E^\sharp=E$.
\end{definition}

\noindent It can easily be shown that $|E|=|E^\sharp|$.\\
In the sequel, by $\{u>c\}$ we mean the set $\{x\in E :u(x)>c\}$. 
\begin{definition}\label{St.f}
Let $E\subset\mathbb{R}^N$ be a measurable set of finite measure and
$u: E \to \mathbb{R}$ a measurable function bounded from below. Then the function $u^\sharp: E^\sharp \to \mathbb{R}$, defined by $u^\sharp(x)=\sup\{c\in \mathbb{R}:x\in\{u>c\}^\sharp\}$,
is said the \emph{Steiner symmetrization} of $u$ with respect to the hyperplane $T$ and $u$ is called  \emph{Steiner symmetric} if $u^\sharp=u$ a.e. in $E$.
\end{definition}
\noindent
It can be proved that 
$\{x\in E^\sharp: u^\sharp(x)>t\}=\{x\in E: u(x)>t\}^\sharp$ for all $t\in \mathbb{R}$. In particular, when $E$ is Steiner symmetric, $u$ and $u^\sharp$ are equimeasurable.

\smallskip
 \noindent We remind some well known properties of the Steiner symmetrization:\\
\emph{a}) if $E\subset\mathbb{R}^N$ is a measurable set of finite measure, 
$u: E \to \mathbb{R}$ is a measurable function bounded from below and $\psi: \mathbb{R}\to\mathbb{R}$ is an increasing
function, then
\begin{equation}\label{crescente}
\psi(u^\sharp)= (\psi(u))^\sharp\quad\text{a.e. in } E
\end{equation}
 (see \cite[Lemma 3.2]{Bro});\\
\emph{b}) if $E\subset\mathbb{R}^N$ is a measurable set of finite measure, $u, v:E \to \mathbb{R}$ are two measurable functions bounded from below such that $uv\in L^1(E)$,
then the \emph{Hardy-Littlewood inequality} holds:
\begin{equation}\label{HL}
\int_E u(x)v(x)\,dx\leq \int_{E^\sharp}u^\sharp(x)v^\sharp(x)\,dx
\end{equation}
(see \cite[Lemma 3.3]{Bro});\\
\emph{c}) if $\Omega$ is a bounded domain and $u\in H_0^1(\Omega)$ is nonnegative, then $u^\sharp\in H_0^1(\Omega^\sharp)$
and the \emph{P\`olya-Szeg\"o inequality} holds:

\begin{equation}\label{PS}
\int_{\Omega^\sharp}|\nabla u^\sharp(x)|^2\,dx\leq\int_\Omega|\nabla u(x)|^2\,dx 
\end{equation}
(see \cite[Theorem 2.1]{CF}).\\
\noindent The proof of Theorem \ref{steiner} relies on a deep result of Cianchi and Fusco which we specialize to our case (see \cite[Theorem 2.6 and Proposition 2.3]{CF}).

\begin{proposition}\label{CF}
Let $\Omega\subset\mathbb{R}^N$ be a Steiner symmetric bounded domain.
Let $u\in H^1_0(\Omega)$ be a nonnegative function
satisfying
\begin{equation}\label{7b}
\left|\left\{(x_1,x')\in\Omega:u^\sharp_{x_1}(x_1,x')=0\right\}\cap
\big\{(x_1,x')\in \Omega:u^\sharp(x_1,x')<M(x')\big\}\right|=0.
\end{equation}
where $M(x')={\rm esssup} \big\{ u^\sharp(x_1,x'): (x_1,x')\in \Omega\cap l(x')\big\}$.
If equality is attained in \eqref{PS},  
then $u^\sharp=u$ a.e. in $\Omega$. 
\end{proposition}
\begin{proof}[Proof of Theorem \ref{steiner}]
 Let $\check  m$ be a minimizer of problem \eqref{teo2}; by Proposition \ref{CCP}, there exists an increasing function $\psi$  such that $\check{m}=\psi(u_{\check m})$, where $u_{\check{m}}$ denotes the unique positive eigenfunction normalized by $\|u_{\check m}\|_{H_0^1(\Omega)}= 1$ of \eqref{A} with $m=\check{m}$ and $\lambda=
\lambda_1(\check m)$. Therefore, by property \eqref{crescente}, the Steiner symmetry of $\check{m}$ is an immediate consequence of the
Steiner symmetry of $u_{\check m}$. 
Hence, we only need to show that $u_{\check m}^\sharp=u_{\check m}$.
By using \eqref{D} with $m=\check m$, $\lambda=\lambda_1(\check m)$, $u=u_{\check m}$ and letting $\varphi\to u_{\check m}$ in
$H_0^1(\Omega)$ we find
\begin{linenomath}
\begin{equation*}
\check\lambda_1=\lambda_1(\check m)=\cfrac{\int_\Omega |\nabla u_{\check m}|^2\,dx}{\int_\Omega \check m u_{\check m}^2
\,dx}\,.
\end{equation*}
\end{linenomath}
The inequalities \eqref{HL}, \eqref{PS} and property \eqref{crescente} yield 
\begin{linenomath}
\begin{equation*}
 \int_\Omega \check m u_{\check m}^2\,dx\leq\int_\Omega \check m^\sharp (u_{\check m}^\sharp)^2
\,dx
\quad\text{and}\quad
\int_\Omega |\nabla u_{\check m}|^2\,dx\geq \int_\Omega |\nabla u_{\check m}^\sharp|^2\,dx.
\end{equation*}
\end{linenomath}
Consequently we deduce
\begin{linenomath}
\begin{equation*}
\check\lambda_1=\cfrac{\int_\Omega |\nabla u_{\check m}|^2\,dx}{\int_\Omega \check m u_{\check{m}}^2
\,dx}\geq\cfrac{\int_\Omega |\nabla u_{\check m}^\sharp|^2\,dx}{\int_\Omega \check m^\sharp 
(u_{\check m}^\sharp)^2
\,dx}\,.
\end{equation*}
\end{linenomath}
Exploiting \eqref{F} and the minimality of $\check\lambda_1$ we can write
\begin{equation}\label{rocek}
\frac{1}{\check\lambda_1}=\cfrac{\int_\Omega \check m u_{\check{m}}^2
\,dx}{\int_\Omega |\nabla u_{\check m}|^2\,dx}\leq\cfrac{\int_\Omega \check m^\sharp 
(u_{\check m}^\sharp)^2
\,dx}{\int_\Omega |\nabla u_{\check m}^\sharp|^2\,dx}\leq \cfrac{\int_\Omega \check m^\sharp 
(u_{\check m^\sharp})^2
\,dx}{\int_\Omega |\nabla u_{\check m^\sharp}|^2\,dx}\,=
\frac{1}{\lambda_1(\check m^\sharp)}\leq\frac{1}{\check\lambda_1}.
\end{equation}
Therefore, all the previous inequalities become equalities and yield
\begin{linenomath}
\begin{equation}\label{K}
\int_\Omega \check m u_{\check m}^2\,dx=\int_\Omega \check m^\sharp (u_{\check m}^\sharp)^2
\,dx
\quad\text{and}\quad\int_\Omega |\nabla u_{\check m}|^2\,dx= \int_\Omega |\nabla u_{\check m}^\sharp|^2\,dx;
\end{equation}
\end{linenomath}
furthermore, by \eqref{F}, $u_{\check{m}}^\sharp$ is an eigenfunction associated to $\lambda_1(\check m^\sharp)$. By the simplicity of $\lambda_1(\check m^\sharp)$,
$u_{\check m}^\sharp$ being positive in $\Omega$ and, by \eqref{K},
$\|u_{\check m}^\sharp\|_{H_0^1(\Omega)}=\|u_{\check m}\|_{H_0^1(\Omega)}=1$, we conclude that $u_{\check m}^\sharp=u_{\check m^\sharp}$.\\
For simplicity of notation, we put $v=u_{\check m}^\sharp=u_{\check m^\sharp}$. The second identity of \eqref{K} will give our result provided we show that the hypothesis \eqref{7b} of Proposition \ref{CF} with $u=u_{\check m}$
is satisfied.
The rest of the proof  is devoted to this task.
By \eqref{rocek}, $\check m^\sharp$ is a minimizer of \eqref{teo2}
and $v$ is the normalized positive eigenfunction associated to
$\lambda_1(\check{m}^\sharp)=\check\lambda_1$.
Moreover, by Proposition \ref{CCP}, there exists an 
increasing function $\Psi$ such that $\check m^\sharp=\Psi(v)$.
Thus $v$ satisfies the problem
\begin{equation}\label{10}
\begin{cases}-\Delta v =\check\lambda_1 \Psi(v) v\quad &\text{in } \Omega,\\
v=0 &\text{on } \partial\Omega. \end{cases}  
\end{equation}
Let $\Omega_+=\{(x_1,x')\in \Omega:\, x_1>0\}$ and  $C_{0,+}^\infty({\Omega_+})=\{\varphi\in C_0^\infty({\Omega_+}): \varphi \text{ is nonnegative}\}$. From \eqref{10} in
weak form we have
\begin{linenomath}
$$\int_{\Omega_+} \nabla v\cdot \nabla \varphi_{x_1}\,dx=\check\lambda_1\int_{\Omega_+}
\Psi(v)v\,\varphi_{x_1}\,dx\quad\forall \varphi\in C^\infty_{0,+}({\Omega_+}).$$
\end{linenomath}
Being $v\in W^{2,2}(\Omega)$, we can rewrite the previous equation as
\begin{linenomath}
$$-\int_{\Omega_+} \nabla v_{x_1}\cdot \nabla \varphi\,dx=\check\lambda_1\int_{\Omega_+}
\Psi(v)v\,\varphi_{x_1}\,dx.$$
\end{linenomath}
Adding  $\check\lambda_1\int_{\Omega_+}
\Psi(v)v_{x_1}\,\varphi\,dx$ to both sides and since $v\in C^{1,\beta}(\overline{\Omega})$,
it becomes
\begin{equation}\label{E-1}
-\int_{\Omega_+} \nabla v_{x_1}\cdot \nabla \varphi\,dx+\check\lambda_1\int_{\Omega_+}
\Psi(v)v_{x_1}\,\varphi\,dx=\check\lambda_1\int_{\Omega_+} \Psi(v)(v\,\varphi)_{x_1}\,dx.
\end{equation}
Let us show that $\int_{\Omega_+} \Psi(v)(v\,\varphi)_{x_1}\,dx\geq 0.$
By Fubini's Theorem we get
\begin{equation}\label{E0}
\int_{\Omega_+} \Psi(v)(v\,\varphi)_{x_1}\,dx=
\int_{\mathbb{R}^{N-1}}dx'\int_0^{b(x')} \Psi(v)(v\,\varphi)_{x_1}\,dx_1,
\end{equation}
where $b(x')=|\Omega \cap l(x')|_1/2$.\\
For any fixed $x'\in\mathbb{R}^{N-1}$, let $\alpha=\alpha(x_1)$ be a primitive of $(v\,\varphi)_{x_1}$ on $[0, b(x')]$.
Since $\alpha(x_1)$ is continuous and $\Psi(v)$ is decreasing
with respect to $x_1$, the Riemann-Stieltjes integral $\int_0^{b(x')} \Psi(v)\,d\alpha(x_1)$ is well defined (see \cite[Theorem 7.27 and the subsequent note]{A}).
Moreover, by using \cite[Theorem 7.8]{A} we have
\begin{equation}\label{E1}\quad
\int_0^{b(x')} \Psi(v)(v\,\varphi)_{x_1}\,dx_1=\int_0^{b(x')} \Psi(v)\,d\alpha(x_1).
\end{equation}
By \cite[Theorems 7.31 and 7.8]{A} there exists a point $x_0$ in $[0, b(x')]$ such that
\begin{linenomath}
\begin{equation*}
\begin{split}
-\int_0^{b(x')}\Psi(v)\,d\alpha(x_1)&=-\Psi(v(0,x'))\int_0^{x_0} \,d\alpha(x_1)-\Psi(v(b(x'), x'))\int_{x_0}^{b(x')} \,d\alpha(x_1)\\
&=-\Psi(v(0,x'))\int_0^{x_0}(v\,\varphi)_{x_1}\,dx_1-\Psi(v(b(x'),x'))\int_{x_0}^{b(x')}(v\,\varphi)_{x_1}\, dx_1.\\
\end{split}
\end{equation*}
%\end{linenomath}
Computing the integrals and recalling that $\varphi\in C^\infty_{0,+}({\Omega_+})$, $v$ is positive and $\Psi(v)$ is decreasing, we conclude that
%\begin{linenomath}
$$-\int_0^{b(x')}\Psi(v)\,d\alpha(x_1)=v(x_0,x')\varphi(x_0,x')\left[\Psi(v(b(x'),x'))-\Psi(v(0,x'))\right]\leq 0.$$
%\end{linenomath}
Therefore, by the previous inequality and \eqref{E1} it follows
$\int_0^{b(x')} \Psi(v)(v\,\varphi)_{x_1}\,dx_1\geq 0$
for any $x'\in\mathbb{R}^{N-1}$ and, in turn, from \eqref{E0} we obtain $\int_{\Omega_+} \Psi(v)(v\,\varphi)_{x_1}\,dx\geq 0$.
Hence, by \eqref{E-1}, $v_{x_1}$ satisfies the differential inequality
%\begin{linenomath}
\begin{equation*}
\Delta v_{x_1}+\check\lambda_1\Psi(v)v_{x_1}\geq 0\quad \text{in }\Omega_+
\end{equation*}
%\end{linenomath}
in weak form.
Then, applying \cite[Theorem 2.5.3]{PS} and since $v_{x_1}\leq 0$ in $\Omega_+$, we conclude that either $v_{x_1}\equiv 0$ or
$v_{x_1}<0$. The former would lead to the contradiction $v\equiv 0$ in $\Omega_+$. 
Consequently, we have $v_{x_1}<0$ in $\Omega_+$.
Similarly it can be shown that $v_{x_1}>0$ in $\Omega_-=\{(x_1,x')\in \Omega:\, x_1<0\}$. Thus
%\begin{linenomath}
\begin{equation*}
\left\{(x_1,x')\in\Omega:v_{x_1}(x_1,x')=0\right\}\cap
\big\{(x_1,x')\in \Omega:v(x_1,x')<M(x')\big\}=\emptyset,
\end{equation*}
\end{linenomath}
where $M(x')={\rm esssup} \big\{ v(x_1,x'): (x_1,x')\in \Omega\cap l(x')\big\}$.
Hence, by Proposition \ref{CF} with $u=u_{\check m}$, we find $u_{\check m}^\sharp= u_{\check m}$ and, finally, $\check m^\sharp=\check m$. This proves the theorem.
\end{proof}
\begin{remark}
The counterpart of this theorem in the case of the fractional Laplacian operator has been proved in \cite{ACF}.
It is somewhat surprising that, in the fractional setting, the proof is much more simple. 
\end{remark}

\noindent As particular cases we obtain the following two corollaries which give the
symmetry
preservation of minimizers of Theorem \ref{cantrell2} and Theorem 3.9 in \cite{CC}. 

\begin{corollary}\label{coro2}
Let $\Omega\subset\mathbb{R}^N$ be a bounded smooth domain and
assume it is Steiner symmetric with respect to the hyperplane
$T=\{x=(x_1,x')\in \mathbb{R}^N: x_1=0\}$.
Moreover, let $E, G$ the measurable subsets of $\Omega$ of the statement
of Theorem \ref{cantrell2}. Then, $E$ and $G$ are Steiner symmetric relative to $T$. 
\end{corollary}
\begin{proof}
Applying Theorem \ref{steiner}  to the minimizers $\check m$ of $\lambda_1(m)$ over
$\mathcal{G}(m_0)$ in the proof of Theorem \ref{cantrell2} we deduce that $\check m$
is a Steiner symmetric function. By the claim after Definition \ref{St.f}, any superlevel
set of $\check m$ is Steiner symmetric with respect to $T$; this provides the Steiner
symmetry of $E$ and $G$.
\end{proof}
\noindent Finally, when $\Omega$ is a ball, Corollary \ref{coro2} specializes to the
following assertion.

\begin{corollary}
Let $\Omega$ be a ball in $\mathbb{R}^N$ and $m_0\in L^\infty(\Omega)$ such that 
$|\{m_0>0\}|>0$. Moreover, let $E, G$ the measurable subsets of $\Omega$ of the
statement of Theorem \ref{cantrell2}. Then, $E$ and $G$ are balls concentric with
$\Omega$.
\end{corollary}
  
\begin{corollary}\label{cor1}
Let $\Omega\subset\mathbb{R}^N$ be a bounded smooth domain and
assume it is Steiner symmetric with respect to the hyperplane
$T=\{x=(x_1,x')\in \mathbb{R}^N: x_1=0\}$.
Let $\mathcal{M}=\{m(x)\in L^{\infty}(\Omega):-m_2\leq m(x)\leq m_1
\text{ a.e. } \Omega,\ m(x)>0 \text{ on a set of positive measure and } \int_\Omega m(x)\,dx=m_3\}$, with $m_1,m_2$
and $m_3$ constants such that $m_1$ and $m_2$ are positive and $-m_2|\Omega|<m_3
< m_1|\Omega|$. 
Then, every measurable set $E\subseteq\Omega$ such that $\check m=
m_1\chi_E-m_2\chi_{\Omega\setminus E}\in\mathcal{M}$ and  $\lambda_1(\check m)=
\inf\{\lambda_1(m):m\in\mathcal{M}\}$ is Steiner symmetric relative to $T$. 
\end{corollary}

\begin{proof}
It is a straightforward consequence of the proof of Theorem 3.9 in \cite{CC} (see the first line after (3.13) on page 311), Lemma \ref{L},
Theorem \ref{steiner} and the claim after Definition \ref{St.f}. 
\end{proof}

\noindent \textbf{Acknowledgments}.
The authors are partially supported by the research project Evolutive and stationary Partial
Differential Equations with a focus on biomathematics (Fondazione di Sardegna 2019).
The authors are members of GNAMPA (Gruppo Nazionale per l'Analisi Matematica, la 
Probabilit\`a e le loro Applicazioni) of INdAM (Istituto Nazionale di Alta Matematica
``Francesco Severi").

%\bibliography{clafa}{}
%\bibliographystyle{plain}

\end{document}